\documentclass[12pt]{amsart}
\usepackage{amsmath}
\usepackage{amsfonts}
\usepackage{amsthm}
\usepackage{amssymb}
\usepackage[english]{babel}
\usepackage[ansinew]{inputenc}
\usepackage[all]{xy}
\usepackage{color}

\theoremstyle{definition}     

\newtheorem{defi}{Definition}[section]
\newtheorem{remark}[defi]{Remark}
\newtheorem{notation}[defi]{Notation}
\newtheorem{example}[defi]{Example}

\theoremstyle{plain}

\newtheorem{theorem}[defi]{Theorem}
\newtheorem{corollary}[defi]{Corollary}
\newtheorem{lemma}[defi]{Lemma}
\newtheorem{proposition}[defi]{Proposition}
\newtheorem{conjecture}[defi]{Conjecture}

\input xy
\xyoption{all}

\title{Models of torsors over affine spaces}

\author{Marco Antei}

\address{Universidad de Costa Rica, Ciudad universitaria Rodrigo Facio Brenes, Costa Rica.}

\email{marco.antei@ucr.ac.cr}

\author{Jorge A. Esquivel A.}

\address{Universidad de Costa Rica, Ciudad universitaria Rodrigo Facio Brenes, Costa Rica.}

\email{jorge.esquivelaraya@ucr.ac.cr}

\date{}
\begin{document}
\maketitle

\textbf{Abstract.} Let $X:=\mathbb{A}^{n}_{R}$ be the $n$-dimensional affine space over a discrete valuation ring $R$ with fraction field $K$. We prove that any pointed torsor $Y$ over $\mathbb{A}^{n}_{K}$ under the action of an affine finite type group scheme can be extended to a torsor over $\mathbb{A}^{n}_{R}$ possibly after pulling $Y$ back over an automorphism of $\mathbb{A}^{n}_{K}$. The proof is effective. Other cases, including $X=\alpha_{p,R}$, will also be discussed. 
\newline

\textbf{Mathematics Subject Classification. Primary: 14L30, 14L15. Secondary: 11G99.}\\\indent
\textbf{Key words:} torsors, affine group schemes, models.

\tableofcontents
\bigskip

\section{Introduction}
\subsection{Aim and scope}\label{sez:Aim}

Let $S$ be a Dedekind scheme of dimension one and $\eta=Spec(K)$ its generic point; let $X$ be a scheme, $f:X\to S$ a faithfully flat morphism of finite type and $f_{\eta}:X_{\eta}\to \eta$ its generic fiber. Assume we are given a finite $K$-group scheme $G$ and a $G$-torsor $Y\to X_{\eta}$. So far the \emph{problem of extending} the $G$-torsor $Y\to X_{\eta}$ has consisted in finding a finite and flat  $S$-group scheme $G'$ whose generic fibre is isomorphic to $G$ and  a $G'$-torsor $T\to X$ whose generic fibre is isomorphic to $Y\to X_{\eta}$ as a $G$-torsor. Some solutions, from Grothendieck's first ideas until nowadays, are known in some particular relevant cases and are the object of many classical and well known results and more recent papers, see for instance \cite[Expos\'e X]{SGA1}, \cite[\S 3]{Ray3},  \cite[\S 2.4]{Saidi}, \cite[Corollary 4.2.8]{Tos}, \cite{Antei2} and \cite{Antei4}. However a general solution does not exist. Moreover it is known that it can even happen that $G$ does not admit a finite and flat model (see for instance \cite{Mi2}). What is always true is that $G$ admits at least an affine, quasi-finite, flat $R$-group scheme model as an easy consequence of \cite[\S 3.4]{WW}. In this paper we study  the problem of extending torsors under the action of very general $G$, that is when $G$ is only affine and of finite type. This approach has been already used in \cite{AE} where it has been proved that, (at least when $X$ has dimension 2, the case $dim(X)>2$ having a different formulation for which we refer the reader to \cite{AE}) every torsor over $X_{\eta}$ under the action of an affine and flat group scheme can be extended to a torsor over $X$ up to a finite number of N\'eron blow up of $X$ at a closed subscheme of its special fiber. In this paper we focus essentially, but not only, on a precise example, the case when $X$ is the affine space $\mathbb{A}^n_R$, i.e. the $n$-dimensional affine space defined over a discrete valuation ring $R$. In this setting we are able to prove the following result (cf. Theorem \ref{TeoAffSpace} and Corollary \ref{CorAffSpa}): 

\begin{theorem}\label{Teo1}Let $X=\mathbb{A}^n_R$  be the $n$-dimensional affine space and $x=(0, ..., 0)$ its origin. Let $G$ be an affine $K$-group scheme of finite type and $f:Y\to X_{\eta}$ a $G$-torsor pointed in $y\in Y_{x_{\eta}}(K)$. Then, possibly after pulling back $Y$ over an automorphism of $\mathbb{A}^n_K$, there exist a $G'$-torsor $f':Y'\to \mathbb{A}^n_R$, pointed in $y'\in Y_{x}(R)$, extending the given $G$-torsor $Y$.
\end{theorem}

This led us to formulate the following conjecture which we are not able to prove at the moment\footnote{Of course in characteristic 0 this statement is empty.}:

\begin{conjecture}\label{Conj1} Let $\pi^{\text{qf}}(\mathbb{A}^n_R, 0)$ denote the quasi-finite fundamental group scheme of $\mathbb{A}^n_R$ at its origin as defined in \cite{AEG}, then the following faithfully flat morphism
$$\pi(\mathbb{A}^n_K, 0)\to \pi^{\text{qf}}(\mathbb{A}^n_R, 0)\times_R K$$ is an isomorphism.
\end{conjecture}   

This conjecture is known to be true if we replace $\mathbb{A}^n_R$ by an abelian scheme, or, more in general, for smooth projective schemes $X$ (with some extra assumptions) over $R$ provided we consider the abelianization $\pi^{\text{ab}}(X,x)$ of the fundamental group scheme $\pi(X,x)$ (cf. \cite{Antei2}).

Similar techniques show that an analog of Theorem \ref{Teo1} may indeed be stated for other interesting cases, for instance when $X=\alpha_{p,R}$. This has a particular interest because reduced scheme are often not studied in this contest. And this is nothing but a single not reduced point. This leads to a conjecture similar to \ref{Conj1} stated in terms of the pseudo-fundamental group scheme, as defined in \cite{AD}.

\indent \textbf{Acknowledgements} Marco Antei would like to thank Michel Emsalem and Arijit Dey for interesting discussions on the subject.

\subsection{Notations and conventions}
\label{sez:Conve}  Let $S$ be any scheme, $X$ a $S$-scheme, $G$ an affine (faithfully) flat $S$-group scheme and $Y$ a $S$-scheme endowed with a right action $\sigma:Y\times G\to Y$. A $S$-morphism $p : Y  \to X$ is said to
be a $G$-torsor if it is affine, faithfully flat, $G$-invariant and the canonical morphism $(\sigma,pr_Y):Y\times G\to Y\times_X Y$ is an isomorphism. Let $H$ be a flat $S$-group scheme and $q:Z\to X$  a $H$-torsor; a morphism  between two such torsors is a pair $(\beta,\alpha):(Z,H)\to(Y,G)$ where $\alpha:H\to G$  is a $S$-morphism of group schemes, and $\beta:Z\to Y$ is a $X$-morphism of schemes such that  the following diagram commutes

$$\xymatrix{Z\times H \ar[r]^{\beta\times \alpha} \ar[d]_{H\text{-}action} & Y\times G\ar[d]^{G\text{-}action}\\ Z\ar[r]_{\beta} & Y}$$ (thus $Y$ is isomorphic to the contracted product $Z\times^H G$  through $\alpha$, cf. \cite{DG}, III, \S 4, 3.2). 
In this case we say that $Z$ precedes $Y$.
Assume moreover that $\alpha$ is a closed immersion. Then $t$ is a closed immersion too and we say that $Z$ is a subtorsor of $Y$ (or that $Z$ is contained in $Y$, or that $Y$ contains $Z$). 

Let  $q\in S$ be any point. For any $S$-scheme $T$ we will denote by $T_q$ the  fiber $T\times_S Spec(k(q))$ of $T$ over $q$. In a similar way for any $S$-morphism of schemes $v:T\to T'$ we will denote by $v_q:T_q\to T'_q$ the reduction of $v$ over $Spec(k(s))$.
When $S$ is irreducible  $\eta$ will denote its generic point and $K$ its function field $k(\eta)$. Any $S$-scheme whose generic fibre is isomorphic to $T_{\eta}$ will be called a model of $T_{\eta}$.   Furthermore when $v_{\eta}$ is an isomorphism we will often say that $v$ is a model map. When $S$ is the spectrum of a discrete valuation ring then $s\in S$ will always denote the special point.

Throughout the whole  paper a  morphism of schemes $f:Y\to X$ will be said to be quasi-finite if it is of finite type and for every point $x\in X$ the fiber $Y_x:=Y\times_X Spec(k(x))$ is a finite set. Let $S$ be any scheme and $G$ an affine $S$-group scheme. Then we say that $G$ is  a finite (resp. quasi-finite/ algebraic) $S$-group scheme if the structural morphism $G\to S$ is finite, (resp. quasi-finite/ of finite type).  

A $G$-torsor $f:Y\to X$ is said to be  finite (resp. quasi-finite/ algebraic) if $G$ is a flat $S$-group scheme which is moreover finite (resp. quasi-finite/ finite type)
Of course when $S$ is the spectrum of a field a $S$-group scheme is quasi-finite if and only if it is finite.

\section{N\'eron blow ups and applications}
\label{sez:blowTo}
\subsection{N\'eron blow ups of torsors}
In this section we recall the notions of N\'eron blow up and its applications in order to N\'eron blow up torsors. This technique in practice provides a useful tool to build new torsors from old ones. As an application we will use this construction to describe all the torsors (cf. Proposition \ref{propScopTorsQuFini}) under a particular quasi-finite group scheme with generic fibre of order $p$ and special fibre of order $1$, using the well known description for some  finite torsors of order $p$. Unless stated otherwise, from now till the end of section \ref{sez:blowTo} we only consider the following situation:

\begin{notation}\label{notaQF} We denote by $S$ the spectrum of a discrete valuation ring $R$  with uniformising element  $\pi$ and with fraction and residue field respectively denoted by $K$ and $k$. As usual $\eta$ and $s$ will denote  the generic and special point of $S$ respectively. Finally we denote by $X$ a faithfully flat $S$-scheme of finite type.
\end{notation}

Hereafter we recall a well known result  that will be used later:

\begin{proposition}\label{propNeronScoppio}  Let notations be as in \ref{notaQF}, let $C$ be a closed subscheme of the special fibre $X_s$ of $X$ and let $\mathcal{I}$ be the sheaf of ideals of $\mathcal{O}_X$ defining $C$. Let $X'\to X$ be the blow up of $X$ at $C$ and $u:X^C\to X$ denote its restriction to the open subscheme of $X'$ where $ \mathcal{I}\cdot\mathcal{O}_X$ is generated by $\pi$. Then:
\begin{enumerate}
\item $X^C$ is a flat $S$-scheme, $u$ is an affine model map. 
\item For any flat $S$-scheme $Z$ and for any $S$-morphism $v:Z\to X$ such that $v_k$ factors  through $C$, there exists a unique $S$-morphism $v':Z\to X^C$ such that $v=u\circ v'$.
\end{enumerate}
\end{proposition}
\begin{proof}Cf. \cite{BLR}, \S 3.2 Proposition 1 or \cite{ANA}, II, 2.1.2 (A).
\end{proof}

The morphism $X^C\to X$ (or simply $X^C$) as in Proposition \ref{propNeronScoppio} is called the N\'eron blow up of $X$ at $C$ and property 2 is often referred to as the universal property of the N\'eron blow up.

Now we are going to explain how to N\'eron blow up torsors:

\begin{lemma}\label{lemmaScoppioTorsori} Let  $G$ be an affine, algebraic and flat $S$-group scheme and $H$ a closed subgroup scheme of $G_s$. Let $Y$ be a $G$-torsor over $X$ and $Z$ a $H$-torsor over $X_s$, subtorsor of $Y_s\to X_s$. Then there exist a faithfully flat $S$-scheme of finite type $X'$, and a model map $\lambda: X'\to X$ such that $Y^Z\to X'$ is a $G^H$-torsor generically isomorphic to $Y_{\eta}\to X_{\eta}$. If moreover $G$ is quasi-finite then $\lambda$ can be obtained from $X$ after a finite number of N\'eron blow ups.
\end{lemma}
\proof This is \cite[Proposition 3.7]{AE}. 
\endproof

The importance of the previous construction is that we can build new torsors from old ones. In order to use this construction we need the special fibre of our given torsor to properly contain some other torsors. This happens, for instance, when the special fibre is trivial, like in the following example: 

\begin{example}\label{esempTorsQuFini2}
Assume $R$ has positive characteristic $p$. Let $X:=Spec(R[x])$ be the affine line over $R$. Then  $$Y:=Spec(R[x,y]/(y^p-y-\pi x))$$ is a non trivial $(\mathbb{Z}/p\mathbb{Z})_R$-torsor (\cite{Mi}, III, Proposition 4.12), with special  fibre $$Y_s= Spec(k[x,y]/(y^p-y))$$ which is a  trivial $(\mathbb{Z}/p\mathbb{Z})_k$-torsor. It is then clear that  $X_s$ is a subtorsor of $Y_s$ and we can blow up $Y$ at $X_s$ following Lemma \ref{lemmaScoppioTorsori} thus getting  a $M$-torsor where $M$ is obtained after N\'eron blowing up $(\mathbb{Z}/p\mathbb{Z})_R$ at $\{1\}_k=Spec(k)$, closed subgroup scheme of $(\mathbb{Z}/p\mathbb{Z})_k$, so that $M=(\mathbb{Z}/p\mathbb{Z})_R^{{\{1\}}_k}=Spec(R[y]/(\pi^{p-1}y^p-y))$; indeed $M=Spec(R[M])$ where $R[M]:=R[x,\pi^{-1}x]/(x^p-x)=R[y]/(\pi^{p-1}y^p-y)$ where we have set $y=\pi^{-1}x$. It is flat as the N\'eron blowing up is always flat, quasi-finite, but clearly not finite. In a similar way $Y^{X_s}=Spec(R[x,y]/(\pi^{p-1}y^p-y-x))$  then we obtain a quasi-finite $M$-torsor. 

\end{example}

In a very similar way we obtain the description of $M$-torsors over an affine scheme: 

\begin{proposition}\label{propScopTorsQuFini}
Assume $R$ has positive characteristic $p$. Let $X:=Spec(A)$ be affine over $R$ with $X_s$ integral. Let $M:=Spec (R[x]/(\pi^{p-1}x^p-x))$ be the $R$-group scheme defined in Example \ref{esempTorsQuFini2}. Then any $M$-torsor over $X$ is isomorphic to a torsor of the form  $$Y:=Spec(A[y]/(\pi^{p-1}y^p-y+a))$$ for some $a\in A$.
\end{proposition}

\begin{proof}
As in Example \ref{esempTorsQuFini2}, if we start from any $(\mathbb{Z}/p\mathbb{Z})_R$-torsor 

$Spec(A[y]/(y^p-y+\pi a))$ and we N\'eron blow it up in $Spec(A_k)\hookrightarrow Y_s$ we obtain the equation $Spec(A[y]/(\pi^{p-1}y^p-y+a))$ which is a $M$-torsor. On the other hand if we start from a $M$-torsor $Y$ over $X$ then one can consider the contracted product $Y\times^M (\mathbb{Z}/p\mathbb{Z})_R$ which is a $(\mathbb{Z}/p\mathbb{Z})_R$-torsor $Z$ with trivial special fibre, so in particular $Y$ is easily seen to be the N\'eron blowing up of $Z$ in $X_s$, hence, as we have just observed, it is isomorphic to $Spec(A[y]/(\pi^{p-1}y^p-y+a))$.
\end{proof}

\section{Extension of torsors}

\label{sez:existence}

Unless stated otherwise, from now till the end of section \ref{sez:existence} we only consider the following situation:

\begin{notation}\label{notaEsiste}
Let $S$ be a \emph{trait}, i.e. the spectrum of a discrete valuation ring $R$ with uniformising element $\pi$, with fraction and residue field denoted by $K$ and $k$ respectively. We denote by $\eta$ and $s$ the generic and special point of $S$. 
\end{notation}

\begin{lemma}\label{lemEsisteModello}
Let notations be as in \ref{notaEsiste} where we assume $X=Spec(A)$ to be affine  and provided with a section $x\in X(R)$. Let $G$ be an affine $K$-group scheme of finite type, $Y=Spec(B)$ a $K$-scheme and $f:Y\to X_{\eta}$ a $G$-torsor pointed in $y\in Y(K)$ lying over $x_{\eta}$. We need the following  technical assumption:
\begin{itemize}
\item
 we fix an embedding $G\hookrightarrow GL_{d,K}$  and we consider the contracted product $Z:=Y\times^G GL_{d,K}$; we assume that $Z\to X_{\eta}$ is a trivial $GL_{d,K}$-torsor (i.e. $Z\simeq GL_{d,X_{\eta}}$).
\end{itemize}
Then there exist a $G'$-torsor $f':Y'\to  X'$ extending the given $G$-torsor $Y$, where   $G'$ is the closure of $G$ in $GL_{d,R}$ and $X'$ is obtained by $X$ after a finite number of N\'eron blow ups of $x_s\in X_s$.
\end{lemma}
\proof
By assumption $X=Spec(A)$ is an affine scheme over $S=Spec(R)$ and we denote by $X_{\eta}=Spec(A_K)$ its generic fibre. 
The point $x$  corresponds to a $R$-ring morphism $\alpha:A\to R$ which, tensoring by $K$ over $R$, gives the $K$-morphism $\alpha_K:A_K\to K$, corresponding to $x_{\eta}$. Since we are assuming that $Y$ has a $K$-rational point $y:Spec(K)\to Y$  over  $x_{\eta}:Spec(K)\to X_{\eta}$ then in particular $Y_{x_{\eta}}=Spec (B\otimes_{A_K}K)\simeq G$  and if we set $C:= B\otimes_{A_K}K$ we can assume $G=Spec(C)$. Hence $C$ is a quotient of $B$ and we have the following commutative diagrams:
\begin{equation}\label{diagSigma}
\xymatrix{C & B \ar@{->>}[l]_(.4){q} \\ K \ar@{^{(}->}[u] & A_K\ar@{^{(}->}[u] \ar@{->>}[l]_(.4){\alpha_K}} \qquad \xymatrix{B\ar[r]^(.4){\rho_B} \ar@{->>}[d]_{q}& C\otimes_K B \ar@{->>}[d]^{id_C\otimes q}\\ C \ar[r]^(.4){\Delta_C} & C\otimes_K C}
\end{equation}
where $\Delta_C$ is the comultiplication of the $K$-Hopf algebra $C$ and $\rho_B$ is the coaction induced by the (right) action of $\sigma:Y\times G\to Y$ thus giving $B$ a structure of (left) comodule over $C$. Finally $q$ is the morphism induced by the closed immersion $G\hookrightarrow Y$ and we will denote by $\varepsilon_C:C\to K$ and $S_C:C\to C$, respectively, the counit and the coinverse morphisms of $C$.  Now consider  the surjective morphism of $A_K$-algebras induced by the closed immersion of $Y$ into the trivial $GL_{d,X_{\eta}}$-torsor:
$$u:A_K[y_{11}, ... , y_{dd},1/det[y_{ij}]]\to B$$
 then if we identify $B$ with the  quotient $A_K[y_{11}, ... , y_{dd},1/det[y_{ij}]]$ by $ker(u)$  and we  take, via $\alpha_K$, the tensor product over $K$, we obtain

\begin{equation}\label{eqArray}
\begin{array}{cc}

B=&\frac{A_K[y_{11}, ... , y_{dd},1/det[y_{ij}]]}{f_1, ..., f_s}\\ & \\
C= &\frac{K[x_{11},  ..., x_{dd},1/det[x_{ij}]]}{{\alpha_K}_*(f_1), ..., {\alpha_K}_*(f_s)}, \qquad q:y_{ij} \mapsto x_{ij}

\end{array}
\end{equation} For each $i=1,...,s$ we  assume that the polynomials  $f_i$  have coefficients in $A$  
.Consequently the $\alpha_*(f_i)$ have coefficients in $R$. 

From the comultiplication on $Z$ (i.e. $\Delta_Z(y_{ij})=\sum_{r=1}^d y_{ir}\otimes y_{rj}$) we deduce:
\begin{equation}\label{eqFicus}
\rho_B(y_{ij})=\sum_{r=1}^d x_{ir} \otimes y_{rj}
\end{equation}
and consequently
\begin{equation}\label{eqFicus2}
\Delta_C(x_{ij})=\sum_{r=1}^d x_{ir} \otimes x_{rj}.
\end{equation}
Applying to the latter the equality $(\varepsilon_C\otimes id)\Delta_C=id$ and comparing coefficients we get

\begin{equation}\label{eqEpsilon}
\varepsilon_C(x_{ij})=\delta_{ij} 
\end{equation}
Moreover recalling that $\Delta_C(S_C,id)=\varepsilon_C$ we obtain 
$$\delta_{ij}=\sum_{r=1}^dS(x_{ir})x_{rj}$$ thus $S_C(x_{sr})$ is the $(s,r)$-th entry ($s$-th row, $r$-th column) in the $d\times d$ matrix $[x_{ij}]^{-1}$. In particular  $\Delta_C(1/(det[x_{ij}]))= 1/(det[x_{ij}])\otimes 1/(det[x_{ij}])$, since $\Delta_C(det[x_{ij}])= det[x_{ij}]\otimes det[x_{ij}]$.

The isomorphism given by $Y\times G\stackrel{\sim}{\longrightarrow} Y\times_{X_{\eta}}Y, (y,g)\mapsto (y,yg)$ gives rise to the isomorphism

\begin{equation}\label{eqIsoEddai}\Psi: B\otimes_{A_K}B
\stackrel{\sim}{\longrightarrow} C\otimes B\qquad y_{ij}\otimes y_{rs}\mapsto \rho(y_{ij})(1\otimes y_{rs})
\end{equation}

We are going to describe $\Psi^{-1}$. Since of course $\Psi^{-1}(1\otimes y_{ij})=(1\otimes y_{ij})$ it only remains to compute $\Psi^{-1}(x_{ij}\otimes 1)$. We claim that $$\Psi^{-1}(x_{ij}\otimes 1)=\sum_{r=1}^d y_{ir}\otimes H(y_{rj})$$ where, for all $(r,s)\in \{1,...,d\}^2$, $H(y_{rs})$ denotes the $(s,r)$-th entry ($s$-th row, $r$-th column) in the $d\times d$ matrix $[y_{ij}]^{-1}$. Indeed
$$\Psi\left(\sum_{r=1}^d y_{ir}\otimes H(y_{rj})\right)=\sum_{r=1}^d \rho(y_{ir})(1\otimes H(y_{rj}))=$$
$$=\sum_{r=1}^d\left(\sum_{s=1}^d x_{is}\otimes (y_{sr}H(y_{rj}))\right)=$$
$$=\sum_{s=1}^d\left(x_{is}\otimes \sum_{r=1}^d(y_{sr}H(y_{rj}))\right)=\sum_{s=1}^d\left(x_{is}\otimes \delta_{sj}\right)=x_{ij}\otimes 1.$$
Now it is important to observe that $H(y_{rj})=P\left(y_{11},...,y_{dd},\frac{1}{det[y_{ij}]}\right)\in \mathbb{Z}\left[y_{11},...,y_{dd},\frac{1}{det[y_{ij}]}\right]$ so in particular it has coefficients in $R$. So let us set 
\begin{equation}\label{eqBpri}
B':=  \frac{A[y_{11}, ... , y_{dd},1/det[y_{ij}]]}{f_1, ..., f_s}
\end{equation}

In order for $Spec(B')$ to be a torsor over $Spec(A)$ we need indeed $B'$ to be $A$-faithfully flat, so we divide the reminder of the proof in two steps: in the first we explain that if $B'$ is $A$-faithfully flat then $Spec(B')$ is a $Spec(C')$-torsor over $Spec(A)$, where $C':= B'\otimes_A R$; in the second we will describe how to always reduce to this situation up to N\'eron blow up the scheme $X$ in $x_s$, the special fibre of the $R$-valued point of $X$:\\
\newline

\underline{\emph{Step 1}: let us assume that $B'$ is $A$-faithfully flat}:

\noindent thus $C'= \frac{R[x_{11},  ..., x_{dd},1/det[x_{ij}]]}{\alpha_*(f_1), ..., \alpha_*(f_s)}$ is $R$-flat and it becomes a Hopf algebra over $R$ when provided with the comultiplication given by the restriction of $\Delta$ to $C'$:
$$\Delta_{C'}:C'\to C'\otimes C' \qquad x_{ij}\mapsto\sum_{r=1}^{d}x_{ir}\otimes  x_{rj}$$
the coinverse given by
$$S_{C'}:C'\to C' \qquad x_{ij}\mapsto S_{C'}(x_{ij})$$
where $S_{C'}(x_{rs})$ denotes the $(s,r)$-th entry in the matrix $[x_{ij}]^{-1}$, and finally the counity given by
$$\varepsilon_{C'}:C'\to R \qquad x_{ij}\mapsto \delta_{ij}.$$

Moreover $B'$ acquires a structure of (left) comodule over $C'$ when provided with the coaction given by
$$\rho_{B'}:B'\to C'\otimes_R B'\qquad y_{ij}\mapsto\sum_{r=1}^{d}x_{ir}\otimes  y_{rj}.$$
Furthermore the natural morphism 
\begin{equation}\label{eqPsiPrimo}\Psi': B'\otimes_{A}B'
\longrightarrow C'\otimes B'\qquad y_{ij}\otimes y_{rs}\mapsto \rho_{B'}(y_{ij})(1\otimes y_{rs}).\end{equation}
has an inverse given by 
\begin{equation}\label{eqPsiInv}\Psi^{'-1}:C'\otimes B'
\stackrel{\sim}{\longrightarrow} B'\otimes_{A}B'\qquad x_{ij}\otimes y_{uv} \mapsto \sum_{r=1}^d y_{ir}\otimes \left(H(y_{rj})y_{uv}\right)
\end{equation} and it is thus an isomorphism. Setting $G':=Spec(C')$ and $Y':=Spec(B')$ then $G'$ is a $R$-flat group scheme of finite type acting on $Y'$  such that $Y'\to X$ is a $G'$-invariant morphism. Finally inverting arrows in (\ref{eqPsiPrimo}) and (\ref{eqPsiInv}) we obtain the desired isomorphism $$Y'\times G'\stackrel{\sim}{\longrightarrow} Y'\times_{X}Y'$$ so, by definition,  $Y'\to X$ is a $G'$-torsor.\\

\underline{\emph{Step 2}: when $B'$ is not $A$-faithfully flat we N\'eron blow up $X$}:

First,  $A$ being of finite type over $R$, we can write $$A=R[t_1, ... , t_r]/u_1(t_1, ..., t_r), ... , u_m(t_1, ..., t_r)$$; so  we rewrite in a useful way equations (\ref{eqArray}):
\begin{equation}\label{eqArraySmart1}
B= \frac{K[t_1, ..., t_r,y_{11}, ... , y_{dd},1/det[y_{ij}]]}{u_1, ... , u_m,f_1, .., f_s}\end{equation}
where the $u_i=u_i(t_1, ..., t_r), i=1, ... , m$ and 

$f_n=f_n(t_1, ..., t_r,y_{11}, ... , y_{dd}, 1/det[y_{ij}]), n=1, ... , s$  are polynomials with coefficients in $K$. Chasing denominators if necessary we can assume that these  polynomials have coefficients in $R$ with at least one coefficient with valuation equal to $0$. Since $X$ is affine we can also assume, up to a translation, that the point $x\in X(R)$ is the origin so that for  $C$ we obtain the following description: 

\begin{equation}\label{eqArraySmart2}
C= \frac{K[x_{11}, ... , x_{dd},1/det[x_{ij}]]}{{\alpha}_*(f_1), .., {\alpha}_*(f_s)}
\end{equation} 

and moreover for every $n=1, ... , s$, $f_n(t_1, ..., t_r,y_{11}, ... , y_{dd},1/det[y_{ij}])$ can be rewritten as
\begin{equation}\label{eqArrayMiDisincasino}
{\alpha}_*(f_n)(y_{11}, ... , y_{dd},1/det[y_{ij}])+\sum_{l=1}^{L_n}v_{nl}(y_{11}, ... , y_{dd},1/det[y_{ij}])g_{nl}(t_1, ... , t_r)
\end{equation} 

for $L_n\in \mathbb{N}$, where $v_{nl}$ and $g_{nl}$ are polynomials with coefficients in $R$, by the above assumption, and $g_{nl}(0, ... ,  0)=0$. Hence we write $B'$ as follows
\begin{equation}\label{eqArraySmart3}
B'= \frac{R[t_1, ..., t_r,y_{11}, ... , y_{dd},1/det[y_{ij}]]}{u_1, ... , u_m,f_1, ... , f_s}\end{equation}
we can assume that $B'$ is $R$-flat, otherwise we can add other polynomials  $f_{s+1}, ... , f_{s'}$ in $R[t_1, ..., t_r,y_{11}, ... , y_{dd}, 1/det[y_{ij}]]$  cutting the $R$-torsion  (thus making it the only $R$-flat quotient of $A[y_{11}, ... , y_{dd},1/det[y_{ij}]]$  which is isomorphic to $B$ after tensoring with $K$ over $R$ (\cite{EGAIV-2} Lemme 2.8.1.1); finally $C':=B'\otimes_AR$ is as follows 

\begin{equation}\label{eqArraySmart4}
C'= \frac{R[x_{11}, ... , x_{dd},1/det[x_{ij}]]}{{\alpha}_*(f_1), .., {\alpha}_*(f_s)}
\end{equation} 

Now, let $e\in \mathbb{N}$ be a positive integer, we N\'eron blow up $e$ times $X$ in $x_s$, the special fibre of the point $x\in X(R)$ that we are assuming to be the origin. This is equivalent to the following construction: we set 
$$t'_{\gamma}:=\pi^{-e} t_{\gamma}, \qquad {\gamma}=1, ... , r$$ 
and  
$$A':=\frac{R[t'_1, ... , t'_r]}{u'_1(t'_1, ..., t'_r), ... , u'_m(t'_1, ..., t'_r)}$$
where $u'_i$ is obtained by $u_i$ replacing $t_{\gamma}$ with $\pi^{e}t_{\gamma}'$ and dividing it by a suitable power of $\pi$ so that the resulting polynomial has coefficients in $R$ with at least one with valuation zero. If we call $X':=Spec(A')$ then $X'$ is the desired N\'eron blow up of $X$ in $x_s\in X_s$ $e$ times. In a similar way from $B'$ we obtain the $R$-flat algebra $B''$

$$\frac{R[t'_1, ..., t'_r,y_{11}, ... , y_{dd},1/det[y_{ij}]]}{u'_1, ... , u'_m,\{{\alpha}_*(f_n)+\sum_{l=1}^{L_n}v_{nl}g'_{nl}\}'_{n=1, ..., s}}$$
where we have first  obtained $g'_{nl}$  by $g_{nl}$ replacing $t_{\gamma}$ with $\pi^{e}t_{\gamma}'$ and then we have divided by a suitable power of $\pi$ the polynomials ${\alpha}_*(f_n)+\sum_{l=1}^{L_n}v_{nl}g'_{nl}$  thus obtaining $\{{\alpha}_*(f_n)+\sum_{l=1}^{L_n}v_{nl}g'_{nl}\}'$  which now has coefficients in $R$ with at least one with valuation zero. We set $Y'':=Spec(B'')$ (it thus coincides, by construction, with the only closed subscheme of $Y'\times_X X'$ which is $R$-flat and generically isomorphic to $Y$) and $G'':=Y''_x=Spec(C'')$ where $$C''=\frac{R[x_{11}, ... , x_{dd},1/det[x_{ij}]]}{\{{\alpha}_*(f_1)\}', ... , \{{\alpha}_*(f_s)\}'}.$$ For a sufficiently big $e$, the exponent of $\pi$ in the equations $t'_{\gamma}=\pi^{-e} t_{\gamma}$, we have 
$$Y''_s=Spec\left(\frac{k[t'_1, ..., t'_r,y_{11}, ... , y_{dd},1/det[y_{ij}]]}{u'_1, ... , u'_m,\{{\alpha}_*(f_1)\}', ... , \{{\alpha}_*(f_s)\}'} \right)$$
which is isomorphic to ${G''_s}_{X'_s}:={G''_s}\times_k{X'_s}$ and thus faithfully flat over $X'_s$ (note that $G''$ being contained in $GL_{d,R}$ always contains a section and it is thus surjective over $Spec(R)$). By the already mentioned \emph{crit\`ere de platitude par fibres} it follows that $Y''\to X'$ is faithfully flat too (and consequently $G''\to Spec(R)$ is flat)  and this concludes the proof.
\endproof
The proof is made in such a way that we always find a model, for the given torsor, which is trivial on the special fibre. However this has only been made for computational purposes. It is clear that in the proof we may N\'eron blow up too much; so in order to obtain a less trivial model we need to blow up a smaller amount of time. Indeed we can chose the \emph{first} $X'$ where the given torsor has a model. This is in general not caught by the proof.

We now state and prove the main consequence of the previous lemma:

\begin{theorem}\label{TeoAffSpace}
Let $X=\mathbb{A}^n_R=Spec(R[x_1, ..., x_n])$ be the $n$-dimensional affine space and $x=(0, ..., 0)$ its origin. Let $G$ be an affine $K$-group scheme of finite type and $f:Y\to X_{\eta}$ a $G$-torsor pointed in $y\in Y_{x_{\eta}}(K)$. Then there exist a $R$-affine and flat group scheme $G'$, a $G'$-torsor $f':Y'\to  X'$, pointed in $y'\in Y_{x}(R)$, extending the given $G$-torsor $Y$, where $X'$ is obtained by $X$ after a finite number of N\'eron blow ups of $x_s\in X_s$.
\end{theorem}
\begin{proof}
By Quillen-Suslin theorem it is known that every $GL_{n,\mathbb{A}^n_K}$-torsor is trivial, then we apply Lemma \ref{lemEsisteModello}.
\end{proof}

\begin{lemma}\label{lemAffSpa} Let $X=\mathbb{A}^n_R$  be the $n$-dimensional affine space and $x=(0, ..., 0)$ its origin as before. Let again $X'$ be the $R$-scheme obtained by $X$ after a finite number of N\'eron blow ups of $x_s\in X_s$. Then there exists a $R$-isomorphism $\mathbb{A}^n_R\to X'$.
\end{lemma}
\begin{proof}
This is well known, at least for $n=1$, (see for instance \cite{WWO}, proof of Theorem 2.2) but for any $n$ the proof is very similar: we compute $X'$ after $m\in \mathbb{N}$ N\'eron blow ups: if $m=1$ then  $$X'=Spec\left( R[x_1, ..., x_n, y_1, ..., y_n ]/ (x_1-\pi y_1,..., x_n-\pi y_n )\right).$$ 
Hence if $m$ is any natural integer  $$X'=Spec\left( R[x_1, ..., x_n, y_1, ..., y_n ]/ (x_1-\pi^m y_1,..., x_n-\pi^m y_n )\right).$$ The morphism  $$X'\to \mathbb{A}^n_R:=Spec(R[t_1, ..., t_n]) ,\quad t_i\mapsto y_i $$ gives the desired isomorphism. 
\end{proof}

This allows us to improve Theorem \ref{TeoAffSpace} as follows

\begin{corollary}\label{CorAffSpa}Let $X=\mathbb{A}^n_R$  be the $n$-dimensional affine space and $x=(0, ..., 0)$ its origin as before. Let $G$ be an affine $K$-group scheme of finite type and $f:Y\to X_{\eta}$ a $G$-torsor pointed in $y\in Y_{x_{\eta}}(K)$. Then, possibly after pulling back $Y$ over an automorphism of $\mathbb{A}^n_K$, there exist a $G'$-torsor $f':Y'\to \mathbb{A}^n_R$, pointed in $y'\in Y_{x}(R)$, extending the given $G$-torsor $Y$.
\end{corollary}
\proof This is an immediate consequence of Theorem \ref{TeoAffSpace} and Lemma \ref{lemAffSpa}. 
\endproof

\begin{remark}
It is worth observing that one can chose any $R$-point of $\mathbb{A}^n_R$ and move it to the origin with a simple translation.  
\end{remark}

Of course Lemma \ref{lemEsisteModello} has other interesting applications not yet stated, for instance it can be applied to the spectrum of local rings. As an example (among many) we mention $X=Spec(R[x]/x^n)$, which becomes interesting when $n=p=char(R)$ and $X$ is nothing but $\alpha_{p,R}$ where we forget the group structure. We state in this case an analog of Corollary \ref{CorAffSpa}: 

\begin{corollary}\label{CorAffSpa2}Let $X=\alpha_{p,R}$ and $x$ its identity element. Let $G$ be an affine $K$-group scheme of finite type and $f:Y\to X_{\eta}$ a $G$-torsor pointed in $y\in Y_{x_{\eta}}(K)$. Then, possibly after pulling back $Y$ over an automorphism (as scheme) of $\alpha_{p,K}$, there exist a $G'$-torsor $f':Y'\to \alpha_{p,R}$, pointed in $y'\in Y_{x}(R)$, extending the given $G$-torsor $Y$.
\end{corollary}

\end{document}